\documentclass[12pt]{amsart}
\usepackage{amsfonts}
\usepackage{enumerate}
\usepackage{helvet,courier}
\usepackage{mathptmx,amsmath}
\usepackage{type1cm}
\DeclareMathAlphabet{\mathcal}{OMS}{cmsy}{m}{n}
\DeclareMathAlphabet{\mathbbold}{U}{bbold}{m}{n}  
\usepackage{amsthm}
\usepackage{graphicx}       
\usepackage{multicol}
\usepackage{mathrsfs}
\usepackage{amssymb}
\usepackage{verbatim}

\usepackage[usenames,dvipsnames]{xcolor}

\theoremstyle{plain}
\newtheorem{thm}{Theorem}
\newtheorem{theorem}[thm]{Theorem}
\newtheorem{lm}[thm]{Lemma}
\newtheorem{lemma}[thm]{Lemma}
\newtheorem{cor}[thm]{Corollary}
\newtheorem{prop}[thm]{Proposition}

\newcounter{question}
\newcommand{\qt}{%
        \stepcounter{question}%
        \thequestion}
\newcommand{\bq}{\fbox{Q\qt}\ }

\newcounter{typo}

\theoremstyle{remark}
\newtheorem{rmk}{Remark}

\theoremstyle{definition}

\newcommand{\bnu}{\begin{enumerate}}
\newcommand{\enu}{\end{enumerate}}
\newcommand{\bpf}{\begin{proof}}
\newcommand{\epf}{\end{proof}}

\newcommand{\q}{\quad}
\newcommand{\qq}{\qquad}

\newcommand{\sset}{\subset}

\newcommand{\eset}{\emptyset}

\newcommand{\al}{\alpha}
\newcommand{\be}{\beta}
\newcommand{\ga}{\gamma}

\newcommand{\om}{\omega}
\newcommand{\Om}{\Omega}

\newcommand{\la}{\lambda}

\newcommand{\z}{\zeta}
\newcommand{\ep}{\epsilon}

\newcommand{\sg}{\sigma}
\newcommand{\tht}{\theta}
\newcommand{\Tht}{\Theta}

\newcommand{\vp}{\varphi}
\newcommand{\de}{\delta}
\newcommand{\De}{\Delta}

\newcommand{\bbz}{\mathbb{Z}}

\newcommand{\bbr}{\mathbb{R}}

\newcommand{\bfr}{\mathbf{R}}
\newcommand{\intrn}{\int_{\bfr^n}}

\newcommand{\lt}{L^2}

\newcommand{\lo}{L^{1}}
\newcommand{\lp}{L^p}
\newcommand{\lnf}{L^{\nf}}

\newcommand{\rar}{\rightarrow}

\newcommand{\f}{\frac}

\newcommand{\p}{\partial}
\newcommand{\nf}{\infty}

\newcommand{\tf}{\tfrac}
\newcommand{\wh}{\widehat}

\allowdisplaybreaks

\makeindex         

\begin{document}

\title{Rough Bilinear Singular Integrals }

\author{Loukas Grafakos, Danqing He,   Petr Honz{\' \i}k }
\thanks{The first author was supported by the Simons Foundation. The third author  
 was supported by the ERC CZ grant LL1203 of the Czech Ministry of Education}
\thanks{2010 Mathematics Classification Number 42B20, 42B99}
\date{}
\maketitle

\begin{abstract}
We study the rough bilinear singular integral, introduced by Coifman and Meyer \cite{CM2}, 
$$
T_\Om(f,g)(x)=\textup{p.v.} \! \int_{\bbr^{n}}\! \int_{\bbr^{n}}\! |(y,z)|^{-2n} \Om((y,z)/|(y,z)|)f(x-y)g(x-z)  dydz,
$$
when $\Om $ is a function in $L^q(\mathbb S^{2n-1})$ with vanishing integral and $2\le q\le \nf$. When 
$q=\infty$ we obtain 
boundedness for $T_\Om$ from $L^{p_1}(\bbr^n)\times L^{p_2}(\bbr^n)$ to $ L^p(\bbr^n) $
when  $1<p_1,  p_2<\nf$ and $1/p=1/p_1+1/p_2$. 
For $q=2$ we obtain that $T_\Om$ is bounded from $L^{2}(\bbr^n)\times L^{ 2}(\bbr^n)$ to $ L^1(\bbr^n) $. 
For   $q$ between $2$ and infinity we obtain the analogous boundedness 
  on a   set of indices   around the point $(1/2,1/2,1)$.  To obtain our results we introduce a new bilinear technique based on 
tensor-type  wavelet decompositions. 
\end{abstract}

\tableofcontents


\section{Introduction}
Singular integral theory was initiated in the seminal work of   Calder\'on and Zygmund      \cite{CZ1}.  
The study of boundedness of rough singular integrals of convolution type has been an active area of research since the middle of the twentieth century.    Calder\'on and Zygmund  \cite{CZ2}
first studied the rough singular integral 
$$
L_\Om(f)(x) = \textup{p.v.} \int_{\mathbb R^n} \f{\Om(y/|y|)}{|y|^n} f(x-y) \, dy
$$
where $\Om$ is in $L\log L(\mathbb S^{n-1})$ with  mean value zero and showed that $L_\Om$ is bounded on $L^p(\mathbb R^n)$ 
for $1<p<\nf$.  The same conclusion under the less restrictive condition that $\Om$ lies in $H^1(\mathbb S^{n-1})$ 
was obtained by Coifman and   Weiss \cite{CW} and Connett \cite{connett}. 
The weak type $(1,1)$ boundedness of $L_\Om$ when   $n=2$ was established  by
Christ and Rubio de Francia \cite{CR}
 and independently by 
Hofmann \cite{hofmann}. 
 (In unpublished work, Christ and Rubio de Francia extended this result to all dimensions $n\le 7$.)
The weak type $(1,1)$ property of 
$L_\Om$ was proved   by Seeger \cite{seeger} in all dimensions and was
later extended by Tao \cite{tao-ind}
to situations in which there is   no Fourier transform structure. Several questions remain   concerning 
the endpoint behavior of $L_\Om$, such as if  the condition $\Om \in L\log L(\mathbb S^{n-1})$ can be relaxed to 
$\Om \in H^1(\mathbb S^{n-1})$, or merely $\Om \in L^1(\mathbb S^{n-1})$ when $\Om$ is an odd function. 
 On the former there is a partial result of Stefanov \cite{stefanov} but not much is still known about the latter.

The bilinear counterpart of the rough singular integral linear theory is notably more intricate. To fix notation, we 
fix $1< q\le \infty$ and we let
  $\Om$ in $L^{q}(\mathbb S^{2n-1})$ with $\int_{\mathbb S^{2n-1}}\Om\,  d\sg=0$, where
$\mathbb S^{2n-1}$ is the unit sphere in $\mathbb R^{2n}$. 
Coifman and Meyer \cite{CM2}  introduced the bilinear singular integral operator  associated with $\Om$ by 
\begin{equation}\label{Op}
T_\Om(f,g)(x)=\textup{p.v.} \int_{\bbr^{n}} \int_{\bbr^{n}} K(x-y,x-z)f(y)g(z)\, dydz,
\end{equation}
where  $f,  g$ are functions in the Schwartz class $\mathcal S(\bbr^n)$, 
$$
K(y,z)=\Om((y,z)')/|(y,z)|^{2n}\, , 
$$
and $x'=x/|x|$ for $x\in\bbr^{2n}$. General facts about bilinear operators can be found in  \cite[Chapter 13]{Meyer3},    \cite[Chapter 7]{Gra14m}, and \cite{MuSc}. 
If $\Om$ possesses some smoothness, i.e. if is   a function of bounded 
variation on the circle, Coifman and Meyer \cite[Theorem I]{CM2}  showed that $T_\Om$ is bounded 
from $L^{p_1}(\mathbb R)\times L^{p_2}(\mathbb R)$ to $L^p(\mathbb R)$ when $1<p_1,p_2,p<\nf$ and $1/p=1/p_1+1/p_2$. In higher dimensions,  it was   shown Grafakos and Torres \cite{Gra-To},  via a    bilinear $T1$ condition, 
that if $\Om$ a Lipschitz function on $\mathbb S^{2n-1}$,  then $T_\Om$ is bounded 
from $L^{p_1}(\mathbb R^n)\times L^{p_2}(\mathbb R^n)$ to $L^p(\mathbb R^n)$ when $1<p_1,p_2 <\nf$, $1/2<p<\nf$, and $1/p=1/p_1+1/p_2$.
  But if  $\Om$ 
is rough, the situation is significantly more complicated, and the boundedness of $T_\Om$ remained unresolved until this work, except when in situations when it   reduces to the uniform boundedness of bilinear Hilbert transforms. 
If $\Om$ is merely integrable function on $\mathbb S^1$, but is 
  odd, the operator $T_\Om$ is intimately connected with the 
celebrated (directional) bilinear Hilbert transform
$$
\mathcal H_{  \theta_1,\theta_2} (f_1, f_2)(x)=
\int_{-\nf}^{+\nf} f_1(x-t\theta_1) f_2(x-t\theta_2)\, \f{dt}{t}
$$  
(in the direction $ (\theta_1,\theta_2)$), via the relationship 
$$
T_\Om (f_1,  f_2)(x)=   \f12 \int_{  \mathbb S^{2n-1}}
\Om(\theta_1,   \theta_2)  \mathcal H_{  \theta_1,\theta_2} (f_1, f_2)(x)\, 
d( \theta_1,\theta_2) \, . 
$$
The boundedness of $\mathcal H_{  \theta_1,\theta_2}$ was proved by Lacey and Thiele \cite{LT1}, \cite{LT2} 
while the more relevant, for this problem, uniform in $\theta_1,\theta_2$ boundedness of 
$\mathcal H_{  \theta_1,\theta_2}$ was addressed by Thiele \cite{thiele}, Grafakos and Li \cite{grafakos-li}, and Li \cite{li}.  
Exploiting the uniform boundedness of $\mathcal H_{  \theta_1,\theta_2}$, Diestel, Grafakos, Honz\'\i k, Si, and Terwilleger \cite{DGHST} showed that if $n=2$ and the 
even part of $\Om$ lies in $H^1(\mathbb S^1)$,   then $T_\Om$ is bounded from $L^{p_1}(\bbr) \times L^{p_2}(\bbr)$
to $L^p(\bbr)$ when $1<p_1,p_2,p<\nf$, \linebreak $1/p=1/p_1+1/p_2$, and the triple $(1/p_1,1/p_2,1/p)$ lies in the 
open hexagon  described by the conditions:
\begin{equation*} 
\Big| \f{1}{p_1}-\f{1}{p_2}\Big| <\f12 \, ,\qq 
\Big| \f{1}{p_1}-\f{1}{p'}\Big| <\f12\,  ,\qq 
\Big| \f{1}{p_2}-\f{1}{p'}\Big| <\f12\,  .
\end{equation*}
This is exactly the region in which the uniform boundedness of the bilinear Hilbert transforms is currently known. 
It is noteworthy to point out the $T_\Om$  reduces itself to a bilinear Hilbert transform $\mathcal H_{  \theta_1,\theta_2}$, 
  if $\Om$  is  the sum of the pointmasses $\de_{(\tht_1,\tht_2)}+\de_{-(\tht_1,\tht_2)}$ on $\mathbb S^1$.

In this work we provide a proof of    the   
boundedness of $T_\Om$  on $L^p$ for all indices with $p>1/2$ and for all dimensions. 
This breakthrough   is a consequence  of  
the novel  technical  ingredients we employ in this context. We 
build on the work of   Duoandikoetxea and Rubio de Francia \cite{DR}
but our key idea is to decompose the multiplier in terms of a 
  tensor-type  compactly-supported wavelet decomposition  and to use 
  combinatorial arguments to group the different pieces together,  exploiting orthogonality.

The main result of this paper is   the following theorem.
\begin{theorem}\label{Main}
For all $n\ge 1$, if $\Om\in L^{\nf}(\mathbb S^{2n-1})$, then for $T_\Om$ defined in \eqref{Op}, we have
\begin{equation}\label{ppk}
\|T_\Om\|_{L^{p_1}(\bbr^n)\times L^{p_2}(\bbr^n)\rar L^p(\bbr^n)}< \infty
\end{equation}
whenever $1<p_1,\, p_2<\nf$ and $1/p=1/p_1+1/p_2$.
\end{theorem}

In the remaining sections we focus on the proof of this result while in the last section we focus on extensions to the 
case where $\Om $ lies in $L^q(\mathbb S^{2n-1})$ for $q<\nf$. 

Some remarks about our notation in this paper: For $1<q<\nf$ we set $q'=   q/(q-1)$  and for $q=\nf$, we set  
$\nf'=1$. 
We denote the 
the norm of a bounded  bilinear operator $T$ from $X\times Y$ to $Z$ by
$$
\|T\|_{X\times Y\to Z} = \sup_{\|f\|_{X}\le 1} 
\sup_{\|g\|_{Y}\le 1} \|T(f,g)\|_Z\, .
$$
This notation was already used in   \eqref{ppk}.  If $x_1,x_2$ are in  $\mathbb R^{ n} $, then we denote the point   $
 (x_1,x_2)$ in $\mathbb R^{ 2n} $ by $\vec x$. We denote   the set of positive integers by  $\mathbb N$ and we set
$\mathbb N_0=\mathbb N\cup\{0\}$. In the sequel, multiindices in $\mathbb Z^{2n}$ are elements of $\mathbb N_0^{2n}$. 
Finally, we adhere to the standard convention to denote by $C$ a constant 
that depends only on inessential parameters of the problem. 

\section{Estimates of Fourier Transforms of the Kernels}

Let us fix a $q$ satisfying $1<q\le \nf$ and a function 
$\Om \in L^q(\mathbb S^{n-1})$ with mean value zero. 
We fix a smooth function $\alpha$ in $\mathbb R^+$ such that $\alpha(t)=1$ for $t\in (0,1]$, $0<\alpha(t)<1$ for $t\in(1,2)$ and $\alpha(t)=0$ for $t\geq 2.$
  For $(y,z)\in \mathbb R^{2n}$ and $j \in \mathbb Z$ we introduce the function  
$$
\beta_j(y,z)=\alpha(2^{-j}|(y,z)|)-\alpha(2^{-j+1}|(y,z)|).
$$		
We write $\be=\be_0$  and we note that this is a function supported in $[1/2,2]$.  
We denote $\Delta_j$ the Littlewood-Paley operator 
$\Delta_j f= \mathcal F^{-1} (\beta_j \wh f\,).$ Here and throughout this paper $\mathcal F^{-1} $ denotes the inverse Fourier 
transform, which is defined via $\mathcal F^{-1} (g)(x) = \int_{\mathbb R^n} g(\xi) e^{2\pi i x\cdot \xi}d\xi = \wh{g}(-x)$, 
where $\wh{g}$ is the Fourier transform of $g$. 
  We decompose the kernel $K$ as follows: we denote $K^i=\beta_i K$ and we set
$ K^i_j= \Delta_{j-i} K^i$ for $i,j\in \mathbb Z$. Then we write
$$
K=\sum_{j=-\nf}^{\nf} K_j,
$$
 where 
 $$
 K_j=\sum_{i=-\infty}^\infty K_j^i .
 $$
  We also denote $m_j=\widehat{K_j}.$

Then the operator can be written as 
$$
T_\Om(f,g)(x)=\sum_j\int_{\mathbb R^n} \int_{\mathbb R^n}  K_j(x-y,x-z)f(y)g(z)\,  dydz=:\sum_jT_{j}(f,g)(x).
$$

We have the following lemma whose proof is known  (see for instance \cite{Duo}) and is omitted.

\begin{lm}\label{K0size}
Given $\Om \in L^q(\mathbb S^{2n-1})$, $0< \de <1/q'$ and
$\vec \xi = (\xi_1,\xi_2 ) \in \mathbb R^{2n}$ we have we have
$$
|\wh {K^0}(\vec \xi \,  )|\le C\|\Om\|_{L^{q}}\min (| \vec \xi \,  |,|\vec \xi \, |^{-\de})
$$
and for all  multiindices $\al$ in $\mathbb Z^{2n}$ with $\al \neq 0$ we have 
$$
|\p^{\al}\wh{K^0}(\vec \xi \,)|\le C_\al \|\Om\|_{L^q}\min(1,| \vec \xi \,  |^{-\de})\, .
$$
\end{lm}

The following proposition is a consequence of the  preceding lemma.
\begin{prop}\label{Good}   
Let $1\le p_1,p_2<\nf$ and define $p$ via  
$1/p=1/p_1+1/p_2$. 
Let $\Om\in L^{q}(\mathbb S^{2n-1})$,  $1<q\le \infty$, $0<\de<1/q'$, and for $j\in \mathbb Z $
consider  the bilinear  operator  
$$
T_j(f,g)(x)=\int_{\bbr^{2n}} K_j(x-y,x-z)f(y)g(z)dydz \, .
$$
If both $p_1,p_2>1$, then $T_j$ is bounded 
 from $L^{p_1}(\mathbb R^n) \times L^{p_2}(\mathbb R^n)$ to $L^p(\mathbb R^n)$    
with norm at most    $C\,\|\Om\|_{L^q}\,2^{(2n-\de) j}$ if $j\ge 0$ and 
at most    $C\,\|\Om\|_{L^q}\,2^{-|j|(1-\de)}$ if $j< 0$. 
If at least one of $p_1$ and $p_2$ is equal to $1$, then $T_j$  maps
$L^{p_1}(\mathbb R^n) \times L^{p_2}(\mathbb R^n)$ to $L^{p,\nf}(\mathbb R^n)$ with a similar norm.   
\end{prop}

\begin{proof}
We   prove the assertion by showing that  the multiplier $ \sigma_j = \wh K_j $ associated  with $T_j$ 
satisfies the conditions of the 
  Coifman-Meyer multiplier theorem \cite{CM}, which  was extended to the case 
$p<1$ by  Kenig and Stein \cite{KS}   and by   Grafakos and Torres \cite{Gra-To}.   
To be able to use this theorem, we need to show that $\sigma_j$ is a $C^{\nf}$ function on $\bbr^{2n}\backslash\{0\})$   
that satisfies 
$$
|\p^{\al}\sg_j( \vec \xi\, )|\le C\,\,Q(j) \|\Om\|_{L^q} | \vec \xi  |^{-|\al|}
$$
for all multiindices $\al$ in $\mathbb Z^{2n}$ with $|\al|\le 2n$ and all $ \vec \xi  \in \mathbb R^{2n}\setminus \{0\}$,
where $Q(j)= 2^{(2n-\de)j}$ if  $j\ge 0$ and  $Q(j)= 2^{-|j|(1-\de)}$ if  $j< 0$. 
Then we   may use   Theorem 7.5.3  in \cite{Gra14m} 
 to deduce the claimed boundedness. 
It is not hard to verify that 
\begin{equation}\label{665}
\sg_j( \vec \xi\,)
= \sum_{i=-\nf}^{\nf}\be(2^{i-j}| \vec \xi\,|)\wh{K^0}(2^i \vec \xi\, )
\end{equation}
If $| \vec \xi\, |\approx 2^l$, then since $\be $  is supported in $[1/2,2]$, 
 $2^i $ must be comparable to $2^{j-l}$ in \eqref{665}. Using Lemma \ref{K0size} we have the 
estimate
$$
|\sg_j( \vec \xi\,)|\le\sum_{i \in F}|\wh{K^0}(2^i \vec \xi\,)|
\le C \|\Om\|_{L^q} \sum_{i \in F } \min\big\{ 2^i| \vec \xi  |, ( 2^i| \vec \xi  |)^{-\de}\big\}
\le C \|\Om\|_{L^q}I(j), 
$$
where $F$ is a finite set of $i$'s near $j-l$ and $I(j)=2^{-|j|\de}$ if $j\ge 0$ whereas $I(j)=2^{-|j| }$ if $j<0$. 
For an $\al$th derivative  of $\sg_j$ with  $1\le |\al|\le 2n$, using 
 that $|\p^{\al}\wh{K^0}( \vec \xi\, )|\le C_{\al}\,\|\Om\|_{L^q} | \vec \xi |^{-\de}$, we obtain
\begin{eqnarray*}
\sum_{i\in F}|\p^{\al}(\wh{K^0}(2^i \vec \xi\, )\Phi(2^i \vec \xi\, ))|
& \le&  \|\Om\|_{L^q}\sum_{i\in F}C_{\al}2^{i|\al|}(2^i| \vec \xi  |)^{-\de}  \\
& \le & C \|\Om\|_{L^q}2^{j(|\al|-\de)}|\vec \xi |^{-|\al|}
\end{eqnarray*}
and this is at most 
$C\,\|\Om\|_{L^q}\,2^{(2n-\de)j}$ if $j\ge 0$ and 
at most     $C\,\|\Om\|_{L^q}\,2^{-|j|(1-\de)}$ if $j< 0$, since 
 $\de \in (0,1/q')$.
\end{proof}

The   operators $T_j$  associated with the multipliers $\wh{K_j}$ are bounded with bounds that grow in $j$ since the smoothness of the symbol is getting worse with $j$. We certainly  have that  
$$
      \|\widehat K_j\|_{L^\infty}\leq C2^{-|j|\de},            
$$
but there is no good estimate available for the derivatives of $\widehat K_j$,  
and moreover, a good $L^\nf$ estimate for the multiplier  does not suffice to yield boundedness in the bilinear setting. The key argument of this article is to circumvent this obstacle and prove that the norms of the operators $T_j$ indeed decay  exponentially. Our proof is new in this context and is based  on a suitable  wavelet expansion   combined with 
combinatorial arguments.

\section{Boundedness: a good point}

In this section we prove the following result which is  a special case of Theorem \ref{Main}:
\begin{theorem}\label{MM}
Suppose $\Om\in L^{q}(\mathbb S^{2n-1})$ with  $2\le q\le \nf $, then for $f,g$ in $L^2(\bbr^n)$ we have
$$
\|T_\Om(f,g)\|_{L^1(\bbr^n)}\leq C\|\Omega\|_{L^q} \|f\|_{L^2(\bbr^n)}\|g\|_{L^2(\bbr^n)}.
$$
\end{theorem}

In view of Proposition \ref{Good}, Theorem \ref{MM} will be a
consequence of the following proposition. 

\begin{prop}\label{D}
Given $2\le q\le \nf$ 
and   $0 <\de<1/(8q')$, 
then for any   $j\ge 0$, the operator $T_j$ 
associated with the kernel $K_j$ maps $L^2(\mathbb R^n)\times L^2(\mathbb R^n) $ to $ L^{1}(\mathbb R^n)$ with norm at most 
$C\|\Om\|_{L^q}2^{-\de j}$.
\end{prop}

To obtain the proof of the proposition, we utilize wavelets with compact support. 
Their existence is due to Daubechies \cite{Dau} and can also be found  in Meyer's book \cite{Meyer1}. 
For our purposes we need
product type smooth wavelets with compact supports; the  construction  of  
such objects    can be found in Triebel~\cite{Tr1}.   

\begin{lm}
For any $k\in \mathbb N$ there are real compactly supported functions $\psi_F,\psi_M\in C^k(\mathbb R)$
such that,  if $\psi^G$ is defined by 
$$
\Psi^{G}(\vec x\,)=\psi_{G_1}(x_1)\cdots \psi_{G_{2n}}(x_{2n}) 
$$
for   $G=(G_1,\dots, G_{2n})$ in the set     
$$
 \mathcal I :=\Big\{ (G_1,\dots, G_{2n}):\,\, G_i \in \{F,M\}\Big\}  \, , 
$$
then the  family of 
functions 
$$
\bigcup_{\vec \mu \in  \mathbb Z^{2n}}\bigg[  \Big\{   \Psi^{(F,\dots, F)} (\vec  x-\vec \mu  )  \Big\} \cup \bigcup_{\la=0}^\nf
\Big\{  2^{\la n}\Psi^{G} (2^{\la}\vec x-\vec \mu):\,\, G\in \mathcal I\setminus \{(F,\dots , F)\}  \Big\}  
  \bigg]
$$
forms an orthonormal basis of $L^2(\mathbb R^{2n})$, where $\vec x= (x_1, \dots , x_{2n})$.  
\end{lm}



\begin{proof}[Proof of Proposition \ref{D}] 

To obtain the estimate, we first decompose the symbol into dyadic pieces, estimate them separately, and then use orthogonality arguments to put them back together.   
Let us take a look at the the symbol $\widehat K_j^0$ which we denote $m_{j,0}$. 
The classical estimates show that 
\begin{equation}\label{del9}
\|m_{j,0}\|_{L^\infty}=\|\wh{K_j^0}\|_{L^{\infty}}\leq C\|\Om\|_{L^q}2^{-\delta j}, \qq 0<\delta<1/q', 
\end{equation}
while for $2\le q\le\nf$
\begin{equation}\label{del99}
\|m_{j,0}\|_{L^2}=\|\be_j ( \widehat{ \be_0 K })\|_{L^2} \leq C\|\widehat{\be_0 K}\|_{L^2}\leq C\|\Om\|_{L^2}\le
C \|\Om\|_{L^q}.
\end{equation}
We observe that for the case $i\neq 0$ we have the identity $m_{j,i}=\widehat K_j^i=m_{j,0}(2^{i}\cdot)$ from the homogeneity of the symbol, 
and thus $m_{j,i}$ also lies in $L^2$.

We utilize a wavelet transform of $m_{j,0}$. We take the product wavelets   described above, with compact supports and $M$ vanishing moments, where $M$ is a large number to be determined later. Here  we choose   generating functions with support diameter  approximately $1$. The wavelets with the same dilation  factor $2^\la$ have some bounded overlap $N$ independent of $\la$.  Since the inverse Fourier
transform of $m_{j,0}$ is essentially supported in the dyadic annulus of 
  radius $1$, the symbol is smooth and the wavelet transform has  a  nice decay. Precisely 
  with 
  $$
\Psi^{\la,G}_{\vec \mu} (\vec x\, )=  2^{\la n}\Psi^{G} (2^{\la}\vec x -\vec\mu)\, , \qq \vec x\,\in \mathbb R^{2n},
  $$
we have the following result: 
\begin{lm}\label{wavesize}
Using the preceding notation, for any $j\in \mathbb Z$ and $\la \in \mathbb N_0$ we have 
\begin{equation}\label{887}
|\langle \Psi^{\la,G}_{\vec \mu}, m_{j,0}\rangle| \leq C\|\Om\|_{L^q}\, 2^{-\delta j}2^{-(M+1+n)\la} \, ,
\end{equation}
where $M$ is the number of vanishing moments of $\psi_M$ and $\de$ is as in \eqref{del9}. 
\end{lm}
\begin{proof}
Let $\la\ge0$ and $G\in \mathcal I\setminus \{(F,\dots , F)\}$. 
We apply the smoothness-cancellation estimate in Appendix B.2 of \cite{Gra14m} with $\Psi$ being the function 
$\Psi_{\vec \mu}^{\la,G}$,
$L=M+1$, and $\Phi$ being the function $m_{j,0}$.  Then we have the properties

(i) $\int_{\mathbb R^{2n}} \Psi_{\vec \mu}^{\la,G}(\vec x\,)\, \vec x\,^{\be}\, d\vec x\,=0$ for $|\be|\le L-1$,

(ii) $|\Psi_{\vec \mu}^{\la,G}(\vec x\,)|\le \f{C2^{\la n}}{(1+2^\la |\vec x\,-2^{-\la}\vec \mu|)^{M_1}}$,

(iii) For $|\al|=L$, $|\p^\al(m_{j,0})(\vec x\,)|\le \f{C\|\Om\|_{L^q}2^{-j\de}}{(1+2^{-j}|\vec x\,|)^{M_2}}$. To verify this property  we notice that since 
 $\be_0$ is a Schwartz function,  we have 
\begin{align*}
|\p^{\al}(\be_j\wh{K^0}) (\vec x\,) |\le\,\,&\,\,
\sum_{\ga\le\al}2^{-j|\ga|}|\p^{\ga}\be_0(2^{-j}\vec x\,)\p^{\al-\ga}\wh{K^0}(\vec x\,)|\\
\le&\,\, C\|\Om\|_{L^q}\sum_{\ga\le\al}2^{-j|\ga|}\f{2^{-j\de}}{(1+2^{-j}|\vec x\,|)^{M_2}}\\
\le&\,\, C\|\Om\|_{L^q}\f{2^{-j\de}}{(1+2^{-j}|\vec x\,|)^{M_2}},
\end{align*}
where we used  Lemma \ref{K0size}, i.e. the property that $|\p^{\al}\wh{K^0}(\vec x\,)|\le C\|\Om\|_{L^q}|\vec x\,|^{-\de}$ for all multiindices $\al$.  

Thus $\Psi_{\vec \mu}^{\la,G}$ has cancellation and $m_{j,0}$ has appropriate smoothness and so it follows that 
$$
|<\Psi_{\vec \mu}^{\la,G},m_{j,0}>|
 \le C\|\Om\|_{L^q}\f{2^{-j\de}2^{\la n}2^{-\la (L+2n)}}{(1+2^{-j-\la }|\vec \mu|)^{M_2}} 
\le C\|\Om\|_{L^q}2^{-j\de}2^{-\la (M+1+n)}, 
$$
thus \eqref{887} holds. Notice that the constant $C$ is independent of $\vec \mu$.

Next we consider the case $\la=0$ and $G=(F,\dots , F)$. In this case we have $|\Psi_{\vec \mu}^{\la,G}(\vec x\,)|\le \f{C }{(1+  |\vec x - \vec \mu|)^{M_1}}$ and
$| m_{j,0} (\vec x\,)|\le \f{C\|\Om\|_{L^q}2^{-j\de}}{(1+2^{-j}|\vec x\,|)^{M_2}}$. Using the result in Appendix B1 in  \cite{Gra14m} we deduce that
$$
|<\Psi_{\vec \mu}^{\la,G},m_{j,0}>|
\le C\|\Om\|_{L^q}\f{2^{-j\de} }{(1+2^{-j  }|\vec \mu\,|)^{M_2}}
\le C\|\Om\|_{L^q}2^{-j\de}  
$$
and thus \eqref{887} follows in this case as well.  
\end{proof}



The wavelets sharing the same generation index  may be organized into $C_{n,M,N}$ groups so that members of the same group have disjoint supports and are of the same product type, i.e., they have the same index $G\in \mathcal I$. For $1\leq \kappa \leq C_{n,M,N} $
we denote by $D_{\lambda,\kappa}$  one of these groups consisting of wavelets 
 whose supports have  diameters about $2^{-\lambda}$.  
We now have that the wavelet expansion
$$
m_{j,0}=\sum_{\substack{ \lambda\geq 0 \\ 1\leq \kappa \leq C_{n,M,N} }}\sum_{\omega\in D_{\lambda,\kappa}} a_\omega \omega
$$
and $\omega$ all have disjoint supports within the group $D_{\lambda,\kappa}$. For the sequence $a=\{a_{\om}\}$ we get $\|a\|_{\ell^2}\leq C$, 
in view of \eqref{del99}, 
because $\{\omega\}$ is an orthonormal basis. 
Since the $\omega$ are continuous functions and  and bounded by $2^{\la n}$, if we set $b_\omega=\|a_\omega \omega\|_{L^\infty}$, 
 we have 
$$
\|\{b_\omega\}_{\omega\in D_{\lambda,\kappa}}\|_{\ell^2}
\le2^{\la n}\Big(\sum_{\omega\in D_{\lambda,\kappa}} |a_{\om}|^2\Big)^{1/2}
\leq C\|\Om\|_{L^2}2^{n\lambda}.
$$
Clearly we also have  
\begin{equation}\label{boundb}
\|\{b_\omega\}_{\omega\in D_{\lambda,\kappa}}\|_{\ell^\infty} \leq \| \{a_\omega\}_{\omega\in D_{\lambda,\kappa}} \|_{\ell^\nf}2^{n\la} 
\leq C\|\Om\|_{L^q}2^{-\delta j-(M+1)\lambda}.
\end{equation}
Now, we split the group $D_{\lambda,\kappa}$ into three parts. Recall the fixed integer $j$ in the statement of Proposition \ref{D}. We define sets 
$$D^1_{\lambda,\kappa} =\Big\{\omega\in D_{\lambda,\kappa}: a_\om \neq 0, \,\, {\rm{supp } \omega} \subset \{(\xi_1,\xi_2):\,\, 2^{-j}|\xi_1| \leq |\xi_2|\leq 2^{j}|\xi_1|\}\Big\},$$
$$D^2_{\lambda,\kappa} =\Big\{\omega\in D_{\lambda,\kappa}: a_\om \neq 0, \,\, {\rm{supp } \omega} \cap \{(\xi_1,\xi_2):\,\,  2^{-j}|\xi_1| \geq |\xi_2|\} \neq \emptyset\Big\},$$
and 
$$D^3_{\lambda,\kappa} =\Big\{\omega\in D_{\lambda,\kappa}: a_\om \neq 0, \,\, {\rm{supp } \omega} \cap \{(\xi_1,\xi_2):\,\,  2^{-j}|\xi_2| \geq |\xi_1|\} \neq \emptyset\Big\}.$$
These groups are  disjoint for large $j$. 
Notice that
$D_{\la,\kappa}^1\cap D_{\la,\kappa}^2=\emptyset$ is obvious. For 
$D_{\la,\kappa}^2$ and $D_{\la,\kappa}^3$  the worst case is   $\la=0$ 
when we have balls of radius $1$ centered at integers, and $D_{\la,\kappa}^2\cap D_{\la,\kappa}^3=\emptyset$    if  $j $
is sufficiently large, for instance $j \ge 100 \sqrt{n}$ works, since 
if $a_\om \neq 0$, then $\om$ is supported in an annulus  centered at the origin of size about $2^j$.   We are assuming here that 
$j \ge 100 \sqrt{n}$ but notice that for  $j < 100 \sqrt{n}$, Proposition \ref{D} is an easy consequence of Proposition \ref{Good}.

We denote, for $\iota=1,2,3$, 
$$ m_{j,0}^\iota= \sum_{\lambda,\kappa} \sum_{\omega\in {D^\iota_{\lambda,\kappa}}} a_\omega \omega,$$
and define
$$
 m^{\iota}_j =\sum_{k=-\nf}^{\nf} m^{\iota}_{j,k} 
$$
with 
$m^{\iota}_{j,k}(\vec \xi\, )=m_{j,0}^{\iota}(2^k\vec \xi\, )$. We  prove   boundedness
for each piece $m_j^1, m_j^2,m_j^3$. We call $m_j^1$ the diagonal part of $m_j$ and $m_j^2,m_j^3$ the off-diagonal parts of $m_j =\wh {K_j}$. 

\section{The diagonal part}

We first deal with the first group $D^1_{\lambda,\kappa}.$ 
Each $\omega\in D^1_{\lambda,\kappa}$ is of   tensor product type $\omega=\omega_1\omega_2$, therefore, we may index the sequences by two indices $k,l\in \mathbb Z^n$ according to the first and second variables. Thus $\omega_{k,l}=\omega_{1,k}\omega_{2,l}$. Likewise, we index the sequence 
$b=\{b_{(k,l)}\}_{k,l}$.  
Now for $r\ge0$ we define sets
$$
U_r=\{(k,l) \in \mathbb Z^{2n}:2^{-r-1}\|b\|_{\ell^\infty}<|b_{(k,l)}|\leq 2^{-r}\|b\|_{\ell^\infty}\}.
$$
From the $\ell^2$ norm of $b$, we find that the cardinality of this set is at most $C\|\Om\|_{L^2}^22^{2n\lambda}2^{2r}\|b\|_{\ell^\infty}^{-2}.$
Indeed,  we have 
$$
|U_r|\le 4\sum_{(k,l)\in U_r}|b_{(k,l)}|^2(\|b\|_{\ell^\nf}2^{-r})^{-2}\le 4\|b\|_{\ell^2}^2\|b\|_{\ell^\nf}^{-2}2^{2r}\le 
C \f{ \|\Om\|^2_{L^2}} {\|b\|_{\ell^\nf}^{ 2}}   2^{2n\la}2^{2r}.
$$
We split each $U_r=U_r^1\cup U_r^2 \cup U_r^3 $, where 
$$
U_r^1=\{(k,l)\in U_r:{\rm card}\{s: (k,s)\in U_r\}\geq 2^{(r+\de j+M\la)/4} \},
$$
$$
U_r^2=\{(k,l)\in U_r\setminus U_r^1:{\rm card}\{s:(s,l)\in U_r\setminus U_r^1\}\geq 2^{(r+\de j+M\la)/4}\}.
$$
and the third set is the remainder. These three sets are disjoint. We notice that if the index $k$ satisfies 
${\rm card}\{s: (k,s)\in U_r\}\geq 2^{(r+\de j+M\la)/4}$, 
then the pair $(k,l)$ lies in $U_r^1$ for all $l\in \mathbb Z^n$ such that 
$(k,l) \in U_r$. 

We observe that in the first set $U_r^1$, we have 
\begin{equation}\label{N1}
N_1:={\rm card}\{ k: {\rm there \; is\;} l\;{\rm s.t.}\; (k,l)\in U_r^1\}\leq C\f{ \|\Om\|_{L^2}^2}{\|b\|_{\ell^\infty}^{ 2} } 
2^{(2n-M/4)\lambda+\f{ 7r}4- \f {j\de} 4},
\end{equation}
since $N_12^{(r+\de j+M\la)/4}\le C\|\Om\|_{L^2}^22^{2n\la}2^{2r}\|b\|_{\ell^\infty}^{-2}$.
We now write $$m_j^{r,1}=\sum_{(k,l)\in U_r^1} a_{(k,l)}\omega_{1,k}\omega_{2,l}$$ and estimate the norm of $m_j^{r,1}$ as a bilinear multiplier as follows:
$$
\begin{aligned}
\|T_{m_j^{r,1}}(f,g)\|_{L^1}
&\leq \Big\|\sum_{(k,l)\in  U_r^1} a_{(k,l)}{\mathcal F}^{-1}(\omega_{1,k}\widehat f\,){\mathcal F}^{-1} (\omega_{2,l}\widehat g\,)\Big\|_{L^1}
\\&\leq \sum_{k\in E} \|\omega_{1,k}\widehat f\,\|_{L^2} \Big\|\sum_{(k,l)\in U_r^1}a_{(k,l)} \omega_{2,l}\widehat g \,\Big\|_{L^2}.
\end{aligned}
$$

For fixed $k$, by the choice of $D_{\la,\kappa}$, the supports of $\om_{k,l}=\om_{1,k}\om_{2,l}$
are disjoint, in particular, the supports of $\om_{2,l}$ are disjoint.  Since 
$\|\om_{1,k}\|_{L^\nf}\approx2^{\la n/2}$, we have the estimate
$$
\Big\|\!\sum_{(k,l)\in U_{r}^1}a_{(k,l)}\om_{2,l}\Big\|_{L^\nf} \le 
C \Big\|\sum_{(k,l)\in U_r^1}\! |b_{(k,l)}|
2^{-\la n/2}\chi_{E_l} \Big\|_{L^\nf} 
\le C\|b\|_{\ell^\nf}2^{-r}2^{-\la n/2},
$$
where $E_{l}\subset\bbr^n$ is the support of $\om_{2,l}$.
As a result, 
$$
\Big\|\sum_{(k,l)\in U_r^1}a_{(k,l)}\omega_{2,l}\wh g\, \Big \|_{L^2}
\le C\|b\|_{\ell^\nf}2^{-r}2^{-\la n/2}\|g\|_{L^2}.
$$
Now let $E=\{k:\exists\ l\ s.t.\ (k,l)\in U^1_r\}$ and note that $|E| = N_1$.

Notice that the  $\om_{k,l}$ in $U^1_r$ have the following property. If $(k,l) \neq (k',l')$, then the supports
of $\om_{1,k}$ and $\om_{1,k'}$  are disjoint. 
Since the $\om_{1,k}$ satisfy $\|\om_{1,k}\|_{L^\nf}\approx 2^{\la n/2}$ and have disjoint supports,  
we have
\begin{align*}
\big\|&T_{m_j^{r,1}}(f,g)\big\|_{L^1} \\
 \le\,&\,\,\sum_{k\in E} \big\|\omega_{1,k}\wh f\, \big\|_{L^2}  2^{-\la n/2}2^{-r}\|b\|_{\ell^\infty} \|g\|_{L^2}\\
 \le\,&\,\,\Big(\sum_{k\in E}1\Big)^{1/2}\Big(\sum_{k\in E}\big\|\om_{1,k}\wh f\, \big\|_{L^2}^2\Big)^{1/2}2^{-\la n/2}
2^{-r}\|b\|_{\ell^\nf}\|g\|_{L^2}\\
 \le\, &\,\, C\Big(\|\Om\|_{L^2}^22^{(2n-M/4)\lambda+7r/4-\de j/4}\|b\|_{\ell^\infty}^{-2}\Big)^{\f12}2^{\la n/2}\|f\|_{L^2} 2^{-\la n/2}2^{-r}
\|b\|_{\ell^\nf}\|g\|_{L^2}\\
\le\,&\,\, C\|\Om\|_{L^2}2^{(n-M/8)\la-r/8-\de j/8}\|f\|_{L^2}\|g\|_{L^2},  
\end{align*}
where we used \eqref{N1} and \eqref{boundb}. 
This gives sufficient decay in $j$, $r$ and $\lambda$ if $M\ge 16n$. The set $U_r^2$ is handled the same way.

To estimate the set $U_r^3$, we further decompose it into at most
$2^{(r+\de j+M\la)/2} $ disjoint sets $V_s$, such that if $(k,l),(k',l')\in V_s$ then $(k,l)\neq(k',l')$ implies $k\neq k'$ and $l\neq l'$. 
Indeed, by the definition of $U_r^3$, for each $(k,l)$ in it with $k$ fixed  there exist   at most 
$N_2$ pairs $(k,l')$ in $U_r^3$ with $N_2=2^{(r+\de j+M\la)/4}$. Otherwise,
it is in $U^1_r$ and therefore a contradiction. Similarly 
for each $(k,l)$ in  $U_r^3$ with $l$ fixed
we have at most $N_2$  pairs 
$(k',l)$ in $U^3_r$. Therefore 
we have at most $N_2^2= C2^{(r+\de j+M\la)/2}$  sets $V_{s}$
satisfying the claimed property.

For each of these sets, since $|a_{\om}|=C|b_{\om}|2^{-\la n}$,  for the multiplier 
$$
 m_j^{V_s} = \sum_{(k,l)\in  V_s} a_{(k,l)} \omega_{1,k} \omega_{2,l} 
$$
we have the following estimate
$$
\begin{aligned}
\big\|T_{m_j^{V_s}}(f,g)\big\|_{L^1}
&\leq \sum_{(k,l)\in V_s}|b_{(k,l)}|2^{-\la n}\big\|{\mathcal F}^{-1}(\omega_{1,k}\widehat f\,){\mathcal F}^{-1} (\omega_{2,l}\widehat g\,)\big\|_{L^1}
\\& \leq C2^{-r} \|b\|_{\ell^\infty} 2^{-\la n}\Big[\!\sum_{(k,l)\in  V_s} \big\|\omega_{1,k}\widehat f\, \big\|_{L^2}^2\Big]^{\frac 12} \!
\Big[\!\sum_{(k,l)\in  V_s} \big\|\omega_{2,l}\widehat g\, \big\|_{L^2}^2\Big]^{\frac 12}
\\& \leq C2^{-r} \|b\|_{\ell^\infty} \|f\|_{L^2}\|g\|_{L^2}.
\end{aligned}
$$
Summing over $s$ and using   estimate \eqref{boundb} and the fact that $N_2^2= C2^{r /2} \|b\|_{\ell^\infty}^{-1/2 }$ yields
\begin{eqnarray*}
  \big\|T_{m_j^{r,3}}(f,g)\big\|_{L^1} &\leq  &
{ 
N_2^2 2^{-r}\|b\|_{\ell^\nf}\|f\|_{L^2}\|g\|_{L^2} } \\
&\le& { 
C2^{(r+\de j+M\la)/2}2^{-r} \|b\|_{\ell^\infty} \|f\|_{L^2}\|g\|_{L^2} }\\
&\le &C\|\Om\|_{L^q} 2^{(-r-\de j-M\la)/2}\|f\|_{L^2}\|g\|_{L^2},
\end{eqnarray*} 
which is also a  good decay. We then have
\begin{eqnarray*}
\, \big\|T_{m_j^r}(f,g)\big\|_{L^1}& \le 
&{ 
\Big[  \big\|T_{m_j^{r,1}}(f,g)\big\|_{L^1}+\big\|T_{m_j^{r,2}}(f,g)\big\|_{L^1}\Big]+ \big\|T_{m_j^{r,3}}(f,g)\big\|_{L^1} }\\
&\le &\,\, { 
C\|\Om\|_{L^2}2^{(n-M/8)\la}2^{-r/8}2^{-\de j/8}\|f\|_{L^2}\|g\|_{L^2}}\\
&&{ 
\qq\qq\q+C\|\Om\|_{L^q} 2^{(-r-\de j-M\la)/2}\|f\|_{L^2}\|g\|_{L^2}}\\
&\le &\,\, C\|\Om\|_{L^q}2^{(-2\de j-M\la-r)/16}\|f\|_{L^2}\|g\|_{L^2}.
\end{eqnarray*}

 Set $f^j= \mathcal F^{-1} (\widehat{f}\chi_{\{c_1\leq |\xi_1| \leq c_2 2^{j+1}\}})$ and 
 $g^j= \mathcal F^{-1} (\widehat{g}\chi_{\{c_1\leq |\xi_2| \leq c_2 2^{j+1}\}})$ for some suitable constants $c_1,c_2>0$. 
In view of the
preceding estimate
for the piece $ m_{j,0}^1= \sum_{\lambda,\kappa} \sum_{\omega\in {D^1_{\lambda,\kappa}}} a_\omega \omega$,    we have 
\begin{align*}
\|T_{ m_{j,0}^1}(f,g)\|_{L^1} = \,\, & \,\, \|T_{ m_{j,0}^1}(f^j,g^j)\|_{L^1} \\
\le  \,\,&{ 
C\|\Om\|_{L^q} \,\,\sum_{\kappa =1}^{  C_{n,M,N}}\sum_{\la\ge0}\sum_{r\ge 0}2^{(-2\de j-M\la-r)/16}\|f^j\|_{L^2}\|g^j\|_{L^2}}\\
\leq  \,\,& \,\, C\|\Om\|_{L^q}2^{-\de j/8 } \|f^j\|_{L^2}\|g^j\|_{L^2}\, .
\end{align*}

The first  equality was obtained from the support properties of $m_{j,0}^1$, 
which comes from the observation that
$m_{j,0}(\vec \xi \,  )\neq 0$ only if 
$|\vec \xi \, |\approx 2^j$, and that $2^{-j}|\xi_1|\le|\xi_2|\le 2^j|\xi_1|$.
Now recall that $m_{j,k}^1(\vec \xi \, )=m^1_{j,0}(2^k\vec \xi \,  )$,   so
\begin{align*}
\,\,&\,\,T_{m_{j,k}^1}(f,g)(x)\\
=\,\,&\,\,\int_{\bbr^{2n}}m_{j,0}^1(2^k \vec \xi \,  )\wh f(\xi_1)\wh g(\xi_2)e^{2\pi ix\cdot(\xi_1+\xi_2)}d\xi_1d\xi_2\\
=\,\,&\,\,\int_{\bbr^{2n}}m_{j,0}^1(\vec \eta\,)
\wh f(2^{-k}\eta_1)\wh g(2^{-k}\eta_2)e^{2\pi i2^{-k}x\cdot(\eta_1+\eta_2)}
2^{-2kn}d\eta_1d\eta_2.
\end{align*}
Denote by $f_k$ the function whose Fourier transform is $\wh f(2^{-k}\xi_1)$
and $E_{j,k}=\{\xi_1\in \bbr^n:\, c_1 2^{-k }\le|\xi_1|\le c_2 2^{j-k}\}$, then
\begin{align*}
\|T_{m_{j,k}^1}(f,g)\|_{L^1}
=\,\,&\,\,  
2^{-2kn}\|T_{m_{j,0}^1}(f_k,g_k)(2^{-k}\cdot)\|_{L^1} \\
=\,\,&\,\,2^{-kn}\|T_{m_{j,0}^1}(f_k,g_k)\|_{L^1}\\
\le\,\,&\,\,  
C\|\Om\|_{L^q}2^{-kn}2^{-\de j/8}\|\wh f(2^{-k}\cdot)\chi_{E_{j,0}}\|_{L^2}\|\wh g(2^{-k}\cdot)\chi_{E_{j,0}}\|_{L^2}   \\
= \,\,&\,\,  C\|\Om\|_{L^q}2^{-\de j/8} \|\wh f\, \|_{L^2(E_{j,k})}\|\wh g\, \|_{L^2(E_{j,k})}
\, .
\end{align*}
Using this estimate and applying the
Cauchy-Schwarz inequality 
we obtain for the diagonal part
$
m_{j}^1=\sum_{k\in \mathbb Z}m_{j,k}^1
$
the estimate
\begin{align*}
\|T_{ m_{j}^1}(f,g)\|_{L^1}\leq&
\sum_{k=-\nf}^{\nf}\|T_{m_{j,k}^1}(f,g)\|_{L^1}\\
\le \,\,&\,\,{ 
C\|\Om\|_{L^q}2^{-\de j/8}\sum_{k=-\nf}^{\nf}\|\wh f\, \|_{L^2(E_{j,k})}\|\wh g\, \|_{L^2(E_{j,k})} }\\
\le\,\,&\,\,C\|\Om\|_{L^q}2^{-\de j/8}\Big(\sum_k\|\wh f\, \|^2_{L^2(E_{j,k})}\Big)^{\frac12}
\Big(\sum_k\|\wh g\, \|^2_{L^2(E_{j,k})}\Big)^{\frac12}\\
\le\,\,&\,\,{ 
C\|\Om\|_{L^q}2^{-\de j/8}j^{1/2}\|f\|_{L^2} j^{1/2}\|g\|_{L^2} }
\\ =\,\,&\,\, C\|\Om\|_{L^q}j2^{-\de j/8} \|f\|_{L^2}\|g\|_{L^2}
\end{align*}
since $\sum_{k=-\nf}^{\nf}\chi_{E_{j,k}}\le j$. 
This completes  the decay of the first piece $m_j^1$.

\section{The off-diagonal parts}

We now   estimate the off-diagonal parts of the operator,
namely $T_{ m_{j}^2}$ and $T_{ m_{j}^3}$.
To control these two operators, we need the following inequality, 
\begin{equation}\label{SIn}
\|T_{ m_{j}^2}(f,g)+T_{ m_{j}^3}(f,g)\|_{L^1}\le
 C\, \|\Om \|_{L^q}\, 2^{-j\de}\|f\|_{L^2}\|g\|_{L^2}.
\end{equation}
which will be discussed in Lemma \ref{SInL}.

Now we show that the right hand side of \eqref{SIn} is finite.
Let us select a group $D^2_{\lambda,\kappa} $ for some $\kappa$. For $\omega \in D^2_{\lambda,\kappa}$ we have 
the estimate $\|b_\omega\|_{L^\infty} \leq C\|\Om\|_{L^q} 2^{-j\epsilon}2^{-M\lambda}.$ We further divide the 
group $D^2_{\lambda,\kappa}$ into columns $D^{2,a}_{\lambda,\kappa}$
such that all wavelets in a given 
column have the form $\omega=\omega_1\omega_2^a$ with the same
$\omega_2^a$, where $a= (\mu_{n+1}, \dots , \mu_{2n})\in \mathbb Z^{n}$. 
Notice that $\om\in D_{\la,\kappa}^2$ implies that $|\xi_2|\le 2$, and 
each $\om_2^a$ is supported
in the cube 
$$
Q=[2^{-\la}(\mu_{n+1}-c),2^{-\la}(\mu_{n+1}+c)]\times\cdots\times
[2^{-\la}(\mu_{2n}-c),2^{-\la}(\mu_{2n}+c)]  
$$
for some $c\approx 1$.
Therefore, we have at most $C\, 2^{\la n}$
choices of $(\mu_{n+1},\cdots,\mu_{2n})$, i.e. there exist at most $C\, 2^{\la n}$ 
different $\om_2^a$ and $C\, 2^{\la n}$ different columns.

For the multiplier $m^{2,a}_{\lambda,\kappa}$ related to the column of $\om^a_2$,
 we then get 
\begin{align*}
\int_{\bbr^{2n}}&\sum_{\om_1}a_{\om}\om_1(\xi_1)\om_2^a(\xi_2)
\wh f(\xi_1)\wh g(\xi_2)e^{2\pi ix\cdot (\xi_1+\xi_2)}d\xi_1d\xi_2\\
=\,\,&\,\,\Big[\sum_{\om_1}a_{\om}T_{\om_1}(f)(x)\Big]   \Big[T_{\om_2^a}(g)(x)\Big]
\end{align*}
with $\om_2^a(\xi_2)=2^{\la n/2}\om_2(2^{\la}\xi_2-a)$. Notice that 
\begin{align*}
|T_{\om_2}^a(g)(x)|=
\,\,&  
\,\,\bigg|2^{-\la n/2}\int_{\bbr^n}     \mathcal F^{-1}(\om_2 )(2^{-\la}(x-y))g(y)
e^{2\pi i2^{-\la}(x-y)\cdot a}dy \bigg|  \\
\le\,\,&\,\,   
2^{-\la n/2}\int_{\bbr^n}\f{g(y)}{(1+2^{-\la}|x-y|)^M}\, dy \\
\le\,\,&\,\, 2^{\la n/2}\mathcal M(g)(x), 
\end{align*}
where $\mathcal M$ is the Hardy-Littlewood maximal function. 
We define 
$$
m_{a,\la}(\xi_1)=\f{\sum_{\om_1}a_{\om}\om_1(\xi_1)\chi_{2^{j-1}\le|\xi_1|\le 2^{j+1}}}
{2^{-j\de}2^{-(M+1+\f n2)\la}},
$$
and then we have 
$$
\sum_{\om_1}a_{\om}T_{\om_1}(f)(x)=2^{-j\de}2^{-(M+1+\f n2)\la}
\int_{\bbr^n}m_{a,\la}(\xi_1)\wh f (\xi_1)e^{2\pi ix\cdot \xi_1}d\xi_1\, , 
$$
since the supports of $\om_1$'s are disjoint and are all contained in the annulus 
$\{\xi_1:\, 2^{j-1}\le|\xi_1|\le 2^{j+1}\}$. 
In view of \eqref{887} in Lemma \ref{wavesize} we have
$|a_{\om}|\le C_M\|\Om\|_{L^q} 2^{-j\de}2^{-(M+1+n)\la}$ and this combined with 
$\|\om_1\|_{L^\nf}\le C2^{n\la/2} 
$
implies that
$|m_{a,\la}|\le C\|\Om\|_{L^q}\chi_{2^{j-1}\le|\xi_1|\le 2^{j+1}}$. Therefore
for the  multiplier   $m=\sum_{\la}2^{-M\la}\sum_{a}m_{a,\la}$ we have
$$
\|T_m(f)\|_{L^2}
\le C\|\Om\|_{L^q}\|\wh f\chi_{2^{j-1}\le|\xi_1|\le 2^{j+1}}\|_{L^2},
$$
since for each fixed $\la$ there exist at most $C2^{\la n}$ indices $a$.

We now can control $T_{m_{j,0}^2}(f,g)(x)$
by 
$
C2^{-j\ep}\mathcal M(g)(x)T_m(f)(x).
$
Recall that $m_{j,k}^2(\vec \xi\, )=m_{j,0}^2(2^k \vec \xi\,)$, then if $f_k$ is the function whose Fourier transform is $\wh f(2^{-k}\xi_1)$,  we have  
$$
|T_{m_{j,k}^2}(f,g)(x)|\le C 2^{-j\de}2^{-2kn} \mathcal M(g_k)(2^{-k}x) |T_m(f_k)(2^{-k}x)|\, .
$$ 

As a result 
\begin{align*}
\,\,&\,\,\Big\|\Big(\sum_{k\in 5\mathbb Z} |T_{ m_{j,k}^2}(f,g)|^2\Big)^{1/2}\Big\|_{L^1}\\
\le \,\,&\,\,\int_{\bbr^n} \Big(\sum_{k\in 5\mathbb Z}|2^{-j\ep}2^{-kn} \mathcal M(g)(x)T_m(f_k)(2^{-k}x)|^2\Big)^{1/2}dx\\
\le\,\,&\,\, 
C 2^{-j\de}\|\mathcal M(g)\|_{L^2}
\bigg(\int_{\bbr^n}\sum_{k\in 5\mathbb Z}|2^{-kn}T_m(f_k)(2^{-k}x)|^2dx\bigg)^{1/2}  \\
\leq\,\,&\,\, C\|\Om\|_{L^q} 2^{-j\de} \|g\|_{L^2}
\bigg(\sum_{k\in 5\mathbb Z}\int_{\bbr^n}\chi_{\{2^{j+k-1}\le|\xi_1|\le2^{j+k+1}\}}|\wh f(\xi_1)|^2d\xi_1\bigg)^{1/2}\\
\leq\,\,&\,\, C\|\Om\|_{L^q}2^{-j\de}\|f\|_{L^2}\|g\|_{L^2}.
\end{align*}
The estimate for $T_{ m_{j}^3}$ is similar. Thus the proof of 
Proposition \ref{D} will be finished once we establish \eqref{SIn}.
The preceding estimate implies that for $f,g$ in $L^2$ we have 
\begin{equation}\label{kjhrr}
\Big\|\Big(\sum_{k\in 5\mathbb Z} |T_{ m_{j,k}^2}(f,g)|^2\Big)^{1/2}\Big\|_{L^1}<\nf\, 
\end{equation}
a fact that will be useful in the sequel. 

\begin{lemma}\label{SInL}
There is a constant $C $ such that   \eqref{SIn} holds for all $f,g$ in $L^2(\mathbb R^n)$.
\end{lemma}


\begin{proof}
We first show that there exists a polynomial $Q_1$ of $n$ variables such that
$T_{ m_{j}^2}(f,g)-Q_1\in L^1$.

Let $\psi\in \mathcal S(\bbr^n)$ such that $\wh \psi\ge0$ with 
supp$\wh\psi\subset\{\xi:1/2\le|\xi|\le 2\}$ and 
$\sum_{j=-\nf}^{\nf}\wh\psi(2^{-j}\xi)=1$ for $\xi\neq 0$.
Set $\wh\Phi=\sum_{j=-2}^2\wh\psi(2^{-j}\xi)$ and define
$\Delta_k(f)=\mathcal F^{-1} (\wh\Phi(2^{-k}\cdot) \wh f\, ) $. 

For $r=0,1,2,3,4$, define $m^{(r)}_j=\sum_{k\in5\bbz+r}m^2_{j,k}$. We will show that 
  there exists a polynomial  $Q_j^r$ such that
\begin{equation}\label{994455}
\|T_{m_j^{(r)}}(f,g)-Q_j^r\|_{L^1}
\le \Big\|\Big(\sum_{k\in 5\bbz+r}|T_{m_{j,k}^2}(f,g)|^2\Big)^{1/2}\Big\|_{L^1}.
\end{equation}
We prove this assertion only in the case $r=0$ as the remaining cases are similar. 
By Corollary $2.2.10$ in \cite{Gra14m} there is a polynomial $Q_1^0$ such that 
$$
\|T_{m_j^{(0)}}(f,g)-Q_1^0\|_{L^1}\le C\, \Big\|\Big(\sum_{k\in 5\bbz }|\Delta_k(T_{m_{j}^{(0)}}(f,g))|^2\Big)^{1/2}\Big\|_{L^1}.
$$
Notice that
\begin{align}\begin{split}\label{GeD}        
&\Delta_k(T_{m_{j}^{(0)}}(f,g))(x) = \\
& \sum_{l\in 5\bbz}\int_{\bbr^{2n}}\wh{\Phi}(2^{-k}(\xi_1+\xi_2))m_{j,0}^2(2^l\xi)
\wh f(\xi_1)\wh g(\xi_2)e^{2\pi ix\cdot(\xi_1+\xi_2)}d\xi_1d\xi_2.
\end{split}\end{align}
Observe that $m_{j,0}^2(\vec \xi\,)$ is supported in the set
$$
\{(\xi_1,\xi_2):\,\, 2^{j-1}\le|(\xi_1,\xi_2)|\le 2^{j+1},|\xi_1|\ge 2^j|\xi_2|\}
$$
which is a subset of 
$$
  \{(\xi_1,\xi_2):\,\, 2^{j-2}\le |\xi_1+\xi_2|\le 2^{j+2}\},
$$
so $m_{j,0}^2(2^l\vec \xi\,)$ is supported in
$\{(\xi_1,\xi_2):2^{j-l-2}\le |\xi_1+\xi_2|\le 2^{j-l+2}\}$.
The integrand in \eqref{GeD} is nonzero only if $k=j-l$, when
$\wh\Phi(2^{-k}\vec \xi\,)m_{j,0}^2(2^l  \vec \xi\, )=m_{j,0}^2(2^l  \vec \xi\,)$, otherwise
the product is $0$. 
In summary we obtained
\begin{equation}\label{9944}
\sum_{k\in 5\bbz}|\Delta_k(T_{m_{j}^{(0)}}(f,g))|^2
=\sum_{k\in 5\mathbb Z}|T_{ m_{j,k}^2}(f,g)|^2.
\end{equation}

Now \eqref{994455} is a consequence of \eqref{kjhrr} and \eqref{9944}.
Thus, there exist polynomials 
$Q_1, Q_2$ such that 
$T_{ m_{j}^2}(f,g)-Q_1,T_{ m_{j}^3}(f,g)-Q_2\in L^1$. 
Given $f,g$ in $L^2(\mathbb R^n)$, we have already shown that  $T_{ m_{j}^1}(f,g)$ lies in
$L^1$. Moreover, we showed in Proposition \ref{Good} that $T_{j}(f,g)$ lies in $L^1$. These 
facts  imply that $T_{ m_{j}^2}(f,g)+T_{ m_{j}^3}(f,g)$ lies in $  L^1$, and 
therefore $Q_1+Q_2=0$. Hence 
\begin{align*}
\|T_{ m_{j}^2}(f,g)+T_{ m_{j}^3}(f,g)\|_{L^1}
\le\,\, &\,\,\|T_{ m_{j}^2}(f,g)-Q_1\|_{L^1}+\|T_{ m_{j}^3}(f,g)-Q_2\|_{L^1}\\
\le\,\, &\,\,C\, \|\Om \|_{L^q}\, 2^{-j\de}\|f\|_{L^2}\|g\|_{L^2}.
\end{align*}
\end{proof}

This completes the proof of  Proposition \ref{D}.\end{proof}%

 \section{Boundedness everywhere when $q=\nf$}

\begin{prop}\label{InB}
Let $\Om\in L^{\nf}(\mathbb S^{2n-1})$, $1<p_1,p_2<\nf$, 
$1/p=1/p_1+1/p_2$. Then for any given $0<\ep<1$ there is a constant $C_{n,\ep}$ 
such that  
$$
\|T_j\|_{L^{p_1}\times L^{p_2}\rar L^{p}}\le C_{n,\ep}\|\Om\|_{\lnf} 2^{j\ep}
$$
for all  $j\ge 0$.    
\end{prop}


To prove Proposition \ref{InB} we use  
Theorem 3 of \cite{Gra-To}
and Proposition \ref{D}. 
To apply the result in \cite{Gra-To} we need to know that the kernel of $T_j$ is 
of bilinear Calder\'on-Zygmund type with bound $A\le C_{n,\ep}\|\Om\|_{\lnf} 2^{j\ep}$ for any $\ep\in (0,1)$.  This is proved in Lemma \ref{CZPP} below. Assuming this lemma, it follows that
$$
\|T_j\|_{L^{p_1}\times L^{p_2}\rar L^{p}}\le 
C (A+\|T_j\|_{L^2\times L^2\rar L^1})\le C_{n,\ep}\|\Om\|_{\lnf} 2^{j\ep},
$$ 
which yields the claim in Proposition \ref{InB}. 

Recall that a bilinear Calder\'on-Zygmund kernel is a function $L $ defined away from the diagonal in $\mathbb R^{3n}$ 
which, for some $A>0$, satisfies the size estimate
$$
| L(u,v,w)| \le \f{A}{(|u-v|+|u-w|+|v-w|)^{2n}}
$$
and the smoothness estimate 
$$
| L(u,v,w)-L(u',v,w)|  \le \f{A |u-u'|^\ep}{(|u-v|+|u-w|+|v-w|)^{2n+\ep}} 
$$
when 
$$
  |u-u'|\le \f13 ( |u-v|+|u-w| )
$$
(with   analogous conditions in $v$ and $w$). 
Such a kernel is associated with  the bilinear operator
$$
(f,g) \mapsto T_L(f,g)(u)=\textup{p.v.}\int_{\mathbb R^n} \int_{\mathbb R^n}  f(v) g(w) L(u,v,w) \, dvdw\, . 
$$
For the theory of such class of operators we refer to \cite{Gra-To}.
Thus we need to prove the following: 

\begin{lm}\label{CZPP}
Given $\Om\in L^{\nf}(\mathbb S^{2n-1})$ and any $j\in \bbz$, for   any $0<\ep<1$ there is a constant $C_{n,\ep}$ such that 
$$
L(u,v,w) =K_j(u-v,u-w)=\sum_{i\in   \mathbb Z} K_{j}^i(u-v,u-w)
$$
 is a    bilinear Calder\'{o}n-Zygmund kernel
with constant $ A\le C_{n,\ep}\|\Om\|_{\lnf} 2^{|j|\ep}$.
\end{lm}

\begin{proof}
We begin by showing that for
given $x,y\in \mathbb R^{2n}$ with  $|x|\ge 3|y|/2$ 
we have
\begin{equation}\label{smth}
\sum_{i\in \mathbb Z}|K_{j}^i(x-y)-K_{j}^i(x)|\le C_{n,\ep}\|\Om\|_{\lnf} \f{2^{|j|\ep}|y|^{\ep}}{|x|^{2n+\ep}}
\end{equation} 

Assuming  \eqref{smth}, we deduce the smoothness of $K_j(u-v,u-w)$ as follows: 
 
\begin{enumerate}

\item[(a)]  For 
 $u,\ v,\ v',\ w\in \bbr^n$ satisfying $|v-v'|\le\tf13(|u-v|+|u-w|)$ we obtain 
\begin{align*}
|K_j(u-v,u-w)&-K_j(u-v',u-w)| \\
\le&\,\, \sum_{i\in \mathbb Z}|K_j^i(u-v',u-w)-K_j^i(u-v,u-w)|\\
\le&\,\,C_{n,\ep}\|\Om\|_{\lnf}\f{2^{|j|\ep}|v-v'|^{\ep}}{(|u-v|+|u-w|)^{2n+\ep}}\\
\le&\,\,C_{n,\ep}\|\Om\|_{\lnf}\f{2^{|j|\ep}|v-v'|^{\ep}}{(|u-v|+|u-w|+|v-w|)^{2n+\ep}}
\end{align*}
since $|u-v|+|u-w|+|v-w|\le 2(|u-v|+|u-w|)$. 

\item[(b)] 
  For $u,\ u',\ v,\ w\in \bbr^n$ satisfying $|u-u'|\le\tf13(|u-v|+|u-w|)$ we take 
$x=(u-v,u-w)$ and $y=(u-u',u-u')$ in  
 \eqref{smth} to deduce the claimed smoothness.  
 
 \item[(c)] 
 For $u, \ v,\ w, \ w'\in \bbr^n$ satisfying $|w-w'|\le\tf13(|u-v|+|u-w|)$ we take
 $x=(u-v,u-w)$ and $y=(0,w'-w)$. 
 
 \end{enumerate}


We may therefore focus on  \eqref{smth}. This will be a consequence of the following estimate
\begin{equation}\label{diff}
|K_j^i(x-y)-K_j^i(x)|\le 
C_{n,\ep}\|\Om\|_{\lnf}\min \Big(1,\f{|y|}{2^{i-j}}\Big)\f{1}{2^{-i\ep}2^{\min(j,0)\ep}|x|^{2n+\ep}}
\end{equation}
when $|x|\ge 3|y|/2$. Assuming \eqref{diff}   we   prove \eqref{smth} as follows:
We pick an integer $N_3$   such that 
  $(\log_2|y|)+j\le N_3<(\log_2|y|)+j+1$. 

If $j\ge 0$, then for $i$ such that $2^{i-j}\le|y|$, i.e., $i\le N_3$,  we have
\begin{align*}
\sum_{i\le N_3}|K_j^i(x-y)-K_j^i(x)|
\le\,\,&\,\,C_{n,\ep}2^{2(2n+\ep)}\|\Om\|_{\lnf}\sum_{i\le N_3}\f{1}{2^{-i\ep}|x|^{2n+\ep}}\\
=\,\,&\,\,C_{n,\ep}2^{2(2n+\ep)}\|\Om\|_{\lnf}|x|^{-2n-\ep}(2^j|y|)^{\ep}\\
=\,\,&\,\,C_{n,\ep}2^{2(2n+\ep)}\|\Om\|_{\lnf}\f{2^{j\ep}|y|^{\ep}}{|x|^{2n+\ep}}\\
=\,\,&\,\,C_{n,\ep}2^{2(2n+\ep)}\|\Om\|_{\lnf}\f{2^{|j|\ep}|y|^{\ep}}{|x|^{2n+\ep}}. 
\end{align*}

For $j\ge0$ and $i>N_3$, i.e. $2^{i-j}>|y|$,
\begin{align*}
\sum_{i>N_3}|K_j^i(x-y)-K_j^i(x)|
\le\,\,&\,\,C_{n,\ep}2^{2(2n+\ep)}\|\Om\|_{\lnf}\sum_{i>N_3}\f{|y|}{2^{i-j}}\f{1}{2^{-i\ep}|x|^{2n+\ep}}\\
\le\,\,&\,\,C_{n,\ep}2^{2(2n+\ep)}\|\Om\|_{\lnf}|y||x|^{-2n-\ep}2^{j}\sum_{i>N_3}2^{i(\ep-1)}\\
=\,\,&\,\,C_{n,\ep}2^{2(2n+\ep)}\|\Om\|_{\lnf}|y||x|^{-2n-\ep}2^{j}(2^j|y|)^{\ep-1}\\
=\,\,&\,\,C_{n,\ep}2^{2(2n+\ep)}\|\Om\|_{\lnf}\f{2^{|j|\ep}|y|^{\ep}}{|x|^{2n+\ep}} .
\end{align*}

If $j<0$, then for $i\le N_3$,
\begin{align*}
\sum_{i\le N_3}|K_j^i(x-y)-K_j^i(x)|
\le\,\,&\,\,C_{n,\ep}2^{2(2n+\ep)}\|\Om\|_{\lnf}\sum_{i\le N_3}\f1{2^{-i\ep}2^{j\ep}|x|^{2n+\ep}}\\
=\,\,&\,\,C_{n,\ep}2^{2(2n+\ep)}\|\Om\|_{\lnf}\f{2^{-j\ep}}{|x|^{2n+\ep}}\sum_{i\le N_3}2^{i\ep}\\
\le\,\,&\,\,C_{n,\ep}2^{2(2n+\ep)}\|\Om\|_{\lnf}\f{2^{-j\ep}}{|x|^{2n+\ep}}(2^j|y|)^{\ep}\\
=\,\,&\,\,C_{n,\ep}2^{2(2n+\ep)}\|\Om\|_{\lnf}\f{|y|^{\ep}}{|x|^{2n+\ep}}.
\end{align*}

If $j<0$ and $i>N_3$, then
\begin{align*}
\sum_{i>N_3}|K_j^i(x-y)-K_j^i(x)|
\le\,\,&\,\,C_{n,\ep}2^{2(2n+\ep)}\|\Om\|_{\lnf}\sum_{i>N_3}\f{|y|}{2^{i-j}}\f{1}{2^{-i\ep}2^{j\ep}|x|^{2n+\ep}}\\
\le\,\,&\,\,C_{n,\ep}2^{2(2n+\ep)}\|\Om\|_{\lnf}|y||x|^{-n-\ep}2^{j(1-\ep)}\sum_{i>N_3}2^{i(\ep-1)}\\
=\,\,&\,\,C_{n,\ep}2^{2(2n+\ep)}\|\Om\|_{\lnf}|y||x|^{-n-\ep}2^{j(1-\ep)}(2^j|y|)^{\ep-1}\\
=\,\,&\,\,C_{n,\ep}2^{2(2n+\ep)}\|\Om\|_{\lnf}\f{|y|^{\ep}}{|x|^{2n+\ep}}.
\end{align*}
And for $j<0$
$$
\f{|y|^{\ep}}{|x|^{2n+\ep}}\le\f{2^{|j|\ep}|y|^{\ep}}{|x|^{2n+\ep}}.
$$

 This concludes the proof of  \eqref{smth} assuming \eqref{diff}.    
Finally we   prove \eqref{diff}.

We have a decreasing estimate of $K^i(x)$, i.e. for $\ep\in(0,1)$ and $i\in \mathbb Z$
\begin{align}
|K^i(x)|\le\,\,& \,\,\|\Om\|_{\lnf}2^{-2in}\chi_{\f12\le\f{|x|}{2^i}\le2}(x)    \notag \\
\le\,\,& \,\,\|\Om\|_{\lnf}2^{2n+\ep}\f{2^{-2in}}{(1+2^{-i}|x|)^{2n+\ep}}\chi_{\f12\le\f{|x|}{2^i}\le2}(x)  \notag \\
\le\,\,&\,\,2^{2n+\ep}\|\Om\|_{\lnf}\f{2^{-2in}}{(1+2^{-i}|x|)^{2n+\ep}}.\label{Dec}
\end{align}

Then recall the lemma from Appendix B1 of \cite{Gra14m}, by defining $\Psi(x)=\f{1}{(1+|x|)^{2n+1}}$
we have that for $t\in[0,1]$
\begin{align*}
 (|K^i| & *\Psi_{i-j})(x-ty)\\
\le\,\,&\,\,{ 2^{2n+\ep}\|\Om\|_{\lnf}\intrn\f{2^{-2in}}{(1+2^{-i}|z|)^{2n+\ep}}\f{2^{-2(i-j)n}}{(1+2^{-(i-j)}|x-ty-z|)^{2n+1}}dz}\\
\le\,\,&\,\,C_{n,\ep}2^{2n+\ep}\|\Om\|_{\lnf}\f{2^{\min(-i,-2(i-j))n}}{(1+2^{\min(-i,-(i-j))}|x-ty|)^{2n+\ep}}\\
\le\,\,&{ \,\,C_{n,\ep}2^{2(2n+\ep)}\|\Om\|_{\lnf}\f{2^{-2in}2^{2\min(j,0)n}}{(2^{-i}2^{\min(j,0)}|x|)^{2n+\ep}} }\\
\le\,\,&\,\,C_{n,\ep}2^{2(2n+\ep)}\|\Om\|_{\lnf}\f{2^{i\ep}}{2^{\min(j,0)\ep}|x|^{2n+\ep}},
\end{align*}
which gives the first part of \eqref{diff} by 
taking $t=0$ and $1$ since 
$$
|\mathcal F^{-1}(\be_j)(x)|\le C_{\be}2^{2jn}(1+2^{j}|x|)^{-2n-1}= C_{\be}\Psi_j(x).
$$

The other part follows from the previous estimate in the following way
\begin{align*}
|K_j^i(x-y)& -K_j^i(x)| \\
=\,\,&\,\,{  \bigg|\int_{\bbr^{2n}} K^i(z)(\mathcal F^{-1}(\be_{i-j})(x-y-z)-\mathcal F^{-1}(\be_{i-j})(x-z))dz\bigg|}\\
= \,\,&\,\,\bigg| \int_{\bbr^{2n}} K^i(z)\int_0^12^{-2(i-j)n}2^{-(i-j)}(\nabla(\mathcal F^{-1}\be))(\f{x-ty-z}{2^{i-j}})\cdot y\, dtdz\bigg|\\
\le\,\,&\,\,{ \f{C_{n,\ep}|y|}{2^{i-j}}\int_0^1\int_{\bbr^{2n}}|K^i(z)|\f{2^{-(i-j)n}}{(1+2^{j-i}|x-ty-z|)^{2n+1}}dzdt }\\
\le\,\,&\,\,C_{n,\ep}\f{|y|}{2^{i-j}}\int_0^1(|K^i|*\Psi_{i-j})(x-ty)dt\\
\le\,\,&\,\,C_{n,\ep}2^{2(2n+\ep)}\|\Om\|_{\lnf}\f{|y|}{2^{i-j}}
\f{C_{n,\ep}}{2^{-i\ep}2^{\min(j,0)\ep}|x|^{2n+\ep}}.   
\end{align*}

To prove the size condition, notice that by the decreasing estimate \eqref{Dec} 
we have
\begin{align*}
\sum_{i \in \mathbb Z } |K_{j}^i(v,w)|
\le\,\,&{ \,\,\sum_i|\int K^i(v_1,w_1)\be_{i-j}(v-v_1,w-w_1)\, dv_1dw_1| }\\
\le\,\,&{ \,\,\sum_i C_{n,\ep}\f{2^{-2in}}{(1+c_k2^{-i}|(v,w)|)^{2n+\ep}} }\\
\le\,\,&\,\,C_{n,\ep}\sum_{i>N^*}2^{-2in}+C(c_j|(v,w)|)^{-(2n+\ep)}\sum_{i\le N^*}2^{j\ep}\\
\le\,\,&\,\,C_{n,\ep}|(v,w)|^{-2n}
\end{align*}
where $c_j=2^{\min(0,j)}$ and $N^*$ is the number such that $2^{N^*}\approx c_j|(v,w)|$.
Hence
$$
|K_j(u-v,u-w)|\le  \f{ C_{n,\ep} }{ (|u-v|+|u-w|)^{ 2n}}  \le \f{ C_{n,\ep}}{ (|u-v|+|u-w|+|v-w|)^{ 2n}}\, . 
$$
\end{proof}

We improve Proposition \ref{InB} by giving a necessary decay 
 via interpolation. Once this is proved, 
Theorem \ref{Main} follows trivially.
\begin{lm}
Let $\Om\in L^{\nf}(\mathbb S^{2n-1})$,  
$1<p_1,p_2<\nf$ and $1/p=1/p_1+1/p_2$,
then  there exist constants
$\ep_0>0$ and $C_{n,\ep_0}$ 
such that  for all $j\ge 0$ 
we have
$$
\|T_j\|_{L^{p_1}\times L^{p_2}\rar L^{p}}\le C_{n,\ep_0}\|\Om\|_{\lnf} 2^{-j\ep_0}.
$$
\end{lm}
\begin{proof}
For any triple $(\tf1{p_1},\tf1{p_2},\tf1p)$
with $1/p=1/p_1+1/p_2$, 
we can choose two triples 
$\vec P_1=(\tf1{p_{1,1}},\tf1{p_{1,2}},\tf1{q_1})$ and 
$\vec P_2=(\tf1{p_{2,1}},\tf1{p_{2,2}},\tf1{q_2})$
such that $\vec P_1$, $\vec P_2$ and $(\tf12,\tf12,1)$ are not collinear and
the point $(\tf1{p_1},\tf1{p_2},\tf1p)$ is in the convex hull of them.
By Proposition \ref{InB} and Proposition 
\ref{D},
$T_j$ is bounded at $\vec P_1$, $\vec P_2$ with bound
$C_{n,\ep}\|\Om\|_{\lnf}2^{j\ep}$ for any $\ep\in(0,1)$
 and at $(\tf12,\tf12,1)$ with bound $C_{n}\|\Om\|_{\lnf} 2^{-j\de} $ for some fixed $\de<1/8$.
Applying Theorem 7.2.2 in \cite{Gra14m}  
we obtain that 
$$
\|T_j(f,g)\|_{L^p}\le C_{n,\ep_0}\|\Om\|_{\lnf}2^{-j\ep_0}\|f\|_{L^{p_1}}\|g\|_{L^{p_2}} 
$$
for some constant $\ep_0$ depending on $p_1,p_2,p$.
\end{proof}

As an application of Theorem \ref{Main} we derive  the boundedness of the Calder{\'o}n 
commutator in the full range of exponents $1<p_1,p_2<\nf$, a fact proved in \cite{C75}. The Calder{\'o}n commutator is defined 
in \cite{C65} as 
$$
\mathcal C(a,f)(x)=p.v.\int_{\bbr}\f{A(x)-A(y)}{(x-y)^2}f(y)dy,
$$
where $a$ is the derivative of $A$.
It is a well known fact \cite{CM2} that this operator
can be written as 
$$
p.v.\int_{\bbr}\int_{\bbr}K(x-y,x-z)f(y)a(z)dydz
$$
with $K(y,z)=\tf{e(z)-e(z-y)}{y^2}=\tf{\Om((y,z)/|(y,z)|)}{|(y,z)|^2}$,
where $e(t)=1$ if $t>0$ and $e(t)=0$ if $t<0$.
$K(y,z)$ is odd and homogeneous of degree $-2$ whose restriction
on $\mathbb S^1$ is $\Om(y,z)$.
It is easy to check that $\Om$ is odd, bounded and thus it satisfies the hypothesis of
Theorem \ref{Main}.

\begin{cor}  
Given   $1<p_1,p_2<\nf$ with $1/p=1/p_1+1/p_2$ there is a constant $C$ such that 
$$
\|\mathcal C(a,f)\|_{\lp}\le C\|a\|_{L^{p_1}}\|f\|_{L^{p_2}}
$$
is valid for all  functions $f$ and $a$ on the line. 
\end{cor}

\section{Boundedness  of $T_\Om$ when $\Om \in L^q(\mathbb S^{2n-1})$ with $2\le q<\nf$}
 
Let $\mathcal R$ be the rhombus    of all points 
$(\tf1{p_1},\tf1{p_2},\tf1p)$ with $1\le p_1,p_2\le \nf$ and $1/p=1/p_1+1/p_2$. We let 
  $\mathcal B$ be the set of all points $(\tf1{p_1},\tf1{p_2},\tf1p)$ such that either $p_1$ or $p_2$ are equal to $1$ or $\nf$, 
 i.e. $\mathcal B$ is the boundary of $\mathcal R$.

\begin{theorem}\label{All}  
Given any dimension $n\ge 1$, there is a constant $C_n$ and there exists 
 a neighborhood $\mathcal S$ of the point $(\tf12,\tf12,1)$ in $\mathcal R$, whose size is at least
$C_n (q')^{-2}$, such that if $\Om$ lies in $L^q(\mathbb S^{2n-1})$ with $2\le q\le\nf$, then 
$$
\|T_\Om\|_{L^{p_1}\times L^{p_2}\rar L^p}<\nf
$$
for $(\tf1{p_1},\tf1{p_2},\tf1p)\in \mathcal S$.
\end{theorem}

\begin{proof} 
In Proposition \ref{D} we showed that $\|T_{j}\|_{L^2\times L^2\rar L^1}\le C\|\Om\|_{L^q}2^{-j\de}$
with $\de\approx\tf1{q'}$. Consider the point $(\tf12,\tf12,1)$. Find two other points
 $(\tf1{p_{11}},\tf1{p_{12}},\tf1{q_1})$ and $(\tf1{p_{21}},\tf1{p_{22}},\tf1{q_2})$ in the interior of $\mathcal R$ such that 
 these three points are not colinear.

Then if      $(\tf1{p_1},\tf1{p_2},\tf1p)$   lies in the open convex hull of these three points, precisely, if 
$\tf1{p_i}=\tf1{p_{1i}}\eta_1+\tf1{p_{2i}}\eta_2+\tf12\eta_3$ for $i=1,2$,  and 
$\eta_1+\eta_2+\eta_3=1$, then 
 multilinear interpolation (Theorem 7.2.2 in \cite{Gra14m})  yields that 
$$
\|T_{j}\|_{L^{p_1}\times L^{p_2}\rar L^p}
\le C\|\Om\|_{L^q} 2^{j(2n(\eta_1+\eta_2)-\de\eta_3)}.
$$
 Moreover, if  
 $2n(\eta_1+\eta_2)-\de\eta_3<0$, then 
 $  \sum_{j\ge 0}   \|T_{j}\|_{L^{p_1}\times L^{p_2}\rar L^p} \le C\|\Om\|_{L^q}$. 

If $(\tf1{p_{11}},\tf1{p_{12}},\tf1{q_1})$ and $(\tf1{p_{21}},\tf1{p_{22}},\tf1{q_2})$ are 
close and let $\eta_1=\eta_2=\eta$, we roughly have 
$4n\eta-\de(1-2\eta)<0$, from which we get 
$\eta<\f{\de}{4n+2\de}$.
In particular, all points $\vec P=(\tf1{p_1},\tf1{p_2},\tf1p)$ in the set 
$$
\Big\{\vec P=(1-t)(\tf12,\tf12,1)+t\vec B:\,\, 0\le t\le \de/16n,\ \vec B \in \mathcal B\Big\}
$$
are contained in the claimed neighborhood, whose size is comparable to
 $(q')^{-2}$. 
\end{proof}

\begin{rmk}
Theorem \ref{All} is sharp in the following sense. Let $\vec A=(\tf12,\tf12,1)$
and $\vec B_0=(1,1,2)$. By Theorem \ref{All}, the smallest $p $ such that $(\tf1{p_1},\tf1{p_2},\tf1p)$ lies in $\mathcal S$    
 satisfies
$$
\f1{p }=2\cdot \f{2\de}{16n}+(1-\f{2\de}{16n})=1+\f{\de}{8n},
$$
from which $\tf{1}{p }-1=\tf{\de}{8n}\approx \tf1{q'}$.
For the case $n=1$, by the example in \cite{DGHST}, we have the requirement
$\tf1p+\tf1q\le2$, which implies that $\tf1p-1\le \tf1{q'}$.
\end{rmk}

We end this paper by stating two related open problems:

\begin{itemize}
\item[(a)] Given $\Om \in L^q(\mathbb S^{2n-1})$ with $2\le q<\nf$, find the full range of $p_1,p_2,p$ such that $T_\Om$ maps
$L^{p_1}\times L^{p_2}\to L^p$. 

\item[(b)] Is $T_\Om$ bounded when $\Om \in L^q(\mathbb S^{2n-1})$ for $q<2$?
\end{itemize}

\end{document}